\documentclass{amsart}
\usepackage{amsmath,amsthm,amssymb,amsfonts}
\usepackage[colorlinks=true]{hyperref}

\hypersetup{urlcolor=blue, citecolor=blue}

\def\doi#1{   {\href{http://dx.doi.org/#1}
   {{\mdseries\ttfamily DOI}}}}

\usepackage{graphics}

\def\RR{{\mathbb R}}

\renewcommand{\phi}{\varphi}

\newcommand{\R}{\mathbb{R}}

\newcommand{\beeq}{\begin{equation}}\newcommand{\eneq}{\end{equation}}

\newenvironment{prf}{\noindent {\bf Proof.} }{\endprf\par}
\def \endprf{\hfill  {\vrule height6pt width6pt depth0pt}\medskip}

\def\<{\langle}             \def\>{\rangle}
\def\({\left(}                 \def\){\right)}
\numberwithin{equation}{section}
\newtheorem{thm}{Theorem}[section]
 
 \newtheorem{lem}[thm]{Lemma}
 \newtheorem{prop}[thm]{Proposition}
 \newtheorem{defn}[thm]{Definition}

\title[Global existence for non-isothermal ideal gas system]
{Global existence of strong solution to non-isothermal ideal gas system}

\author{Bin Han}
\address{Department of Mathetics, Hangzhou Dianzi University, Hangzhou, 310018, China}
\email{hanbin@hdu.edu.cn}

\author{Ning-An Lai}
\address{College of Mathematics and Computer Science, Zhejiang
Normal University, Jinhua 321004, China}
		\email{ninganlai@zjnu.edu.cn}

\author{Andrei Tarfulea}
\address{Department of Mathematics, Louisiana State University, Baton Rouge, LA 70803, USA}\email{tarfulea@lsu.edu}

\date{\today}

\begin{document}
\maketitle
\begin{abstract}

This paper aims to establish the global existence of strong solutions to a non-isothermal ideal gas model. We first show global well-posedness in the Sobolev space ($H^2(\RR^3)$) by using energy estimates. We then prove the  global well-posedness for small-data solutions in the critical Besov space by using Banach's fixed point theorem.

\end{abstract}
\tableofcontents

\section{Introduction}
\noindent

Starting from a given free energy, Lai-Liu-Tarfulea \cite{LLT} established a general framework for deriving non-isothermal fluid models by combining classical thermodynamic laws and the energetic variational approach(see \cite{GKL, HKL}). As an application, three full non-isothermal systems (the non-isothermal ideal gas, non-isothermal porous media, and non-isothermal generalized porous media equations) are established based on three specific free energies. What is more, under some appropriate assumption on the conductivity coefficient $\kappa_3$, a maximum/minimum principle is developed for the first two models by adapting an idea originally from the work \cite{Tar}. These maximum/minimum principles establish the positivity of the absolute temperature, which implies the thermodynamic consistency of the corresponding models.

However, \cite{LLT} does not address the long time behavior (existence and uniqueness) of the solution to the non-isothermal models mentioned, while it is the core theory for partial differential system. At present, there are many results on the existence and behavior of weak solutions to various non-isothermal fluid models; see \cite{FN2005, Fei2007, FN12, Fei, NoPe} for the Navier-Stokes-Fourier system, which is a powerful generalization of the classical Navier-Stokes equations and is used to model thermodynamic fluid flow, \cite{DL} for the non-isothermal general Ericksen-Leslie system, \cite{HLLL} for the non-isothermal Poisson-Nernst-Planck-Fourier system, and \cite{SL} for the Brinkman-Fourier system with ideal gas equilibrium.

This paper aims to study the global well-posedness of the following non-isothermal ideal gas system in $\R^3$:
\begin{equation}\label{idmodel}
\left \{
\begin{aligned}
&\partial_t\rho=\kappa_1\Delta(\rho\theta),\\
&\kappa_2(\rho\theta)_t-\kappa_1(\kappa_1+\kappa_2)\nabla\cdot
\left(\theta\nabla(\rho\theta)\right)=\nabla\cdot\left(\kappa_3\nabla\theta\right).\\
\end{aligned} \right.
\end{equation}
For the reader's convenience, we briefly sketch the construction of \eqref{idmodel}. As can be seen in the model, the main unknown variables are:
\begin{enumerate}
\item a non-negative measurable function $\rho=\rho(t, x)$ which denotes the mass density;
\item a positive measurable function $\theta=\theta(t, x)$ representing the absolute temperature.
\end{enumerate}
In addition, a vector field $u=u(t, x)$ denoting the velocity field of the fluid will be used as an intermediate variable.

For an ideal gas, we have the following definition of free energy
\begin{equation}\label{freeen}
\Psi(\rho, \theta)=\kappa_1\theta\rho\ln\rho-\kappa_2\rho\theta\ln\theta.
\end{equation}
Then the (specific) entropy of the system, denoted by $\eta$, and the (specific) internal energy, denoted by $e$, are connected to the free energy $\Psi$ by the standard Helmholtz relation (see formula (2.5.26) in the classical book \cite{Daf})
\begin{equation}\label{entropy}
\left \{
\begin{aligned}
 &\eta(\rho, \theta):=-\partial_{\theta}\Psi,\\
  &e(\rho, \theta):=\Psi-\partial_\theta\Psi\theta=\Psi+\eta\theta,\\
&\eta_{\theta}=\frac{\kappa_2\rho}{\theta}.\\
\end{aligned} \right.
\end{equation}
The total energy and total dissipation are then chosen to be
\[
E^\text{total}=\int_{\Omega_t^x}\Psi(\rho, \theta)dx, \quad\quad \mathcal{D}^\text{total}=\frac12\int_{\Omega_t^x}\rho u^2dx.
\]
Employing the energetic variational approach then establishes the following Darcy type diffusion law
\begin{equation}\label{p}
\left \{
\begin{aligned}
  &p=\Psi_\rho\rho-\rho=\kappa_1\rho\theta,\\
  &\nabla p=-\rho u,\\
  &\partial_{\theta}p=\kappa_1\rho.\\
\end{aligned}\right.
\end{equation}

We remark that, according to \cite{Berry, McQ}, the internal energy and pressure are both linearly proportional to the product of temperature and density.
It is easy to verify this fact by combining \eqref{freeen}, \eqref{entropy} and \eqref{p}.\\

Now, we rewrite the internal energy function in terms of the new state variables $\rho, \eta$, yielding
\begin{equation}
e_1(\rho, \eta)=e\left(\rho, \theta(\rho, \eta)\right),
\end{equation}
which then implies
\begin{equation}\label{deentropy}
\left \{
\begin{aligned}
&e_{1\eta}=\theta,~~~e_{1\rho}=\Psi_{\rho},\\
&\nabla p=\rho\nabla e_{1\rho}+\eta\nabla e_{1\eta}.
\end{aligned}\right.
\end{equation}
Recalling the continuity equation for a closed system
\begin{equation}\label{contin}
\begin{aligned}
  \rho_t+\nabla\cdot(\rho u)=0,
\end{aligned}
\end{equation}
combine this with the two classical thermodynamic laws, the first of which relates the rate of change of the internal energy
with dissipation and heat
\begin{equation}\label{firstlaw}
\begin{aligned}
\frac{de}{dt}=\nabla\cdot W+\nabla\cdot q,
\end{aligned}
\end{equation}
where $W$ denotes the amount of thermodynamic work done by the system on its surroundings and $q$ denotes the quantity of energy supplied to the system as heat. The second thermodynamic law describes the evolution of the entropy
\begin{equation}\label{secondlaw}
\begin{aligned}
\partial_t\eta+\nabla\cdot(\eta u)=\nabla\cdot \left(\frac{q}{\theta}\right)+\Delta,
\end{aligned}
\end{equation}
where $\Delta\ge 0$ denotes the rate of entropy production, and Fourier's law yields
\begin{equation}\label{Flaw}
\begin{aligned}
q=\kappa_3\nabla \theta,
\end{aligned}
\end{equation}
where $\kappa_3$ denotes the material conductivity (which may depend on $\rho$ and $\theta$). Combining \eqref{deentropy}, \eqref{contin}, \eqref{firstlaw}, \eqref{secondlaw}, and \eqref{Flaw}, we obtain
\begin{equation}\label{pre1}
\begin{aligned}
&\frac{de_1(\rho, \eta)}{dt}\\
=&e_{1\rho}\rho_t+e_{1\eta}\eta_t\\
=&e_{1\rho}\left(-\nabla\cdot(\rho u)\right)+e_{1\eta}\left(-\nabla\cdot(\eta u)+\nabla\cdot\left(\frac{q}{\theta}\right)+\Delta\right)\\
=&-\nabla\cdot\left(e_{1\rho}\rho u+e_{1\eta}\eta u\right)+\left(\rho\nabla e_{1\rho}+\eta\nabla e_{1\eta}\right)\cdot u+\theta\nabla\cdot\left(\frac{q}{\theta}\right)+\theta\Delta\\
=&\nabla\cdot W+\nabla p\cdot u+\nabla\cdot q-\frac{q}{\theta}\cdot\nabla\theta+\theta\Delta\\
=&\nabla\cdot W-\rho u^2+\nabla\cdot q-\frac{\kappa_3|\nabla\theta|^2}{\theta}+\theta\Delta.\\
\end{aligned}
\end{equation}
Therefore
\begin{equation}\label{delta}
\left \{
\begin{aligned}
W&=-\left(e_{1\rho}\rho+e_{1\eta}\eta\right)u,\\
\Delta&=\frac{1}{\theta}\left(\rho|u|^2+\frac{\kappa_3|\nabla\theta|^2}
{\theta}\right),\\
\end{aligned}\right.
\end{equation}
which in turn gives
\begin{equation}\label{pre2}
\begin{aligned}
\eta_t+\nabla\cdot(\eta u)
=&\eta_\theta(\theta_t+u\cdot\nabla\theta)+\eta_{\rho}(\rho_t+u\cdot\nabla\rho)+\eta\nabla\cdot u\\
=&\eta_\theta(\theta_t+u\cdot\nabla\theta)+\eta_{\rho}\left(-\rho\nabla\cdot u\right)+\eta\nabla\cdot u\\
=&\eta_\theta(\theta_t+u\cdot\nabla\theta)+\left(\eta-\eta_{\rho}\rho\right)\nabla\cdot u\\
=&\eta_\theta(\theta_t+u\cdot\nabla\theta)+\partial_{\theta}p\nabla\cdot u\\
=&\nabla\cdot\left(\frac{q}{\theta}\right)+\Delta\\
=&\nabla\cdot\left(\frac{q}{\theta}\right)+\frac{1}{\theta}
\left(\rho|u|^2+\frac{q\cdot\nabla\theta}{\theta}\right),
\end{aligned}
\end{equation}
which finally yields
\begin{equation}\label{pre3}
\eta_\theta(\theta_t+u\cdot\nabla\theta)+\partial_{\theta}p\nabla\cdot u
=\nabla\cdot\left(\frac{q}{\theta}\right)+\frac{1}{\theta}
\left(\rho|u|^2+\frac{q\cdot\nabla\theta}{\theta}\right).
\end{equation}

Combining \eqref{pre3}, \eqref{entropy} and \eqref{p} alows us to conclude that
\begin{equation}\label{idpre3}
\begin{aligned}
&\frac{\kappa_2\rho}{\theta}(\theta_t+u\cdot\nabla\theta)+\kappa_1\rho\nabla\cdot u\\
=&\nabla\cdot\left(\frac{\kappa_3\nabla\theta}{\theta}\right)+\frac{1}{\theta}
\left(-\kappa_1\nabla(\rho \theta)\cdot u+\frac{\kappa_3|\nabla\theta|^2}{\theta}\right),\\
\end{aligned}
\end{equation}
so that
\begin{equation}\label{idpre5}
\begin{aligned}
\kappa_2(\rho\theta)_t-\kappa_1(\kappa_1+\kappa_2)\nabla
\cdot\left(\theta\nabla(\rho\theta)\right)=\nabla\cdot
\left(\kappa_3\nabla\theta\right),
\end{aligned}
\end{equation}
which completes the derivation of the non-isothermal ideal gas model \eqref{idmodel}.\\
\\
Our main goal is to establish well-posedness for the system \eqref{idmodel}. Motivated by similar works on the classical Navier-Stokes equations (\cite{Da6, FK}), we first motivate our choice of working spaces. We observe that (\ref{idmodel}) is invariant under the transformation
\begin{align}
	\begin{split}\label{5.1}
		&(\rho(t,x), \theta(t,x))\longrightarrow(\rho(\lambda^2 t,\lambda x),
		\theta(\lambda^2 t,\lambda x)),\\
		&(\rho_0(x),\theta_0(x))\longrightarrow(\rho_0(\lambda x),\theta_0(\lambda x)).
	\end{split}
\end{align}
\begin{defn}\label{d1.1}
	A function space $E\subset\mathcal{S}'(\RR^3)\times\mathcal{S}'
	(\RR^3)$
	is called a critical space if the associated norm
	is invariant under the transformation (\ref{5.1}).
\end{defn}
Obviously $\dot{H}^{3/2}\times\dot{H}^{3/2}$
is a critical space for the initial data, but $\dot{H}^{3/2}$ is not included in $L^\infty$.
We cannot expect to get $L^\infty$ control on the density
and the temperature by taking $(\rho_0-1,\theta_0-1)\in \dot{H}^{3/2}\times \dot{H}^{3/2}$.
Moreover, the product between functions does not
extend continuously from
$\dot{H}^{3/2}\times\dot{H}^{3/2}$
to $\dot{H}^{3/2}$, so that we will run into
difficulties when estimating the nonlinear terms. Similar to the
Navier-Stokes system studied in \cite{Da6},
we could use homogeneous Besov spaces $\dot{B}^s_{2,1}(\RR^3)$
(defined in \cite{BCD}, Chapter 2). $\dot{B}_{2,1}^{3/2}$ is an algebra embedded
in $L^\infty$ which allows us to control the
density and temperature from above without requiring more
regularity on derivatives of $\rho_0$ and $\theta_0$.

Our first result proves global well-posedness for \eqref{idmodel} when the initial data is close to a stable equilibrium $(\underline\rho,\underline\theta)$ in the subcritical space $H^2\times H^2.$ The working space $X(T)$ is defined by the norm
  $$\aligned
\|u\|_{X(T)}&:=\sup\limits_{0\leq \tau\leq T}\|u(\tau)\|^2_{H^2}+\int_0^T	\Big(\|\nabla u\|^2_{H^2}+\|\partial_tu\|^2_{H^1}\Big)d\tau,
\endaligned$$
for any distribution $u$ and $T>0.$
 \begin{thm}\label{t1.2}
	Let $\underline\rho$, $\underline\theta>0$ be fixed constants. There exist two positive constants $c$ and $M$ such that for all $\rho_0$ and $\theta_0$ where $(\rho_0-\underline \rho, \theta_0-\underline\theta)\in H^2\times H^2$ and
	\begin{align}\label{4.1}
		\begin{split}
			\|\rho_0 - \underline\rho\|_{H^2}+	\|\theta_0 - \underline\theta\|_{H^2}\leq c,
		\end{split}
	\end{align}
the system (\ref{idmodel}) has a unique global solution $(\rho,\theta)$ with $(\rho - \underline\rho,\theta - \underline\theta) \in X(T)$ for all $T>0$.
Moreover, if we define $\widetilde\rho := \rho-\underline\rho$ and $\widetilde\theta := \theta - \underline\theta$, then
	\begin{align}\label{4.2}
	\begin{split}
	\|(\widetilde\rho,\widetilde\theta)\|_{X(T)}\leq\frac{1}{2}c.
	\end{split}
\end{align}
	\end{thm}
Based on the above scaling analysis, it will suffice to prove an $H^2$ energy estimate for $\rho$ and $\theta$.

 Our second result then establishes the existence and uniqueness of a solution to the system \eqref{idmodel} for initial data close to a stable equilibrium $(\underline\rho,\underline\theta)$  in the critical space $\dot B_{2,1}^{3/2}\times \dot B_{2,1}^{3/2}$. For convenience, we assume that $\underline\rho=\underline\theta=1$. The working space $E(t)$ is then defined by
 $$E(T):=\left\{u\in \mathcal C\left([0,T],\dot B_{2,1}^{3/2}\right),\quad \nabla^2u\in L^1\left(0,T; \dot B_{2,1}^{3/2}\right)\right\},\quad T>0.$$
 \begin{thm}\label{t1.3}
	There exist two positive constants $c$ and $M$ such  that for all $(a_0,\widetilde\theta_0)\in \dot B_{2,1}^{3/2}\times \dot B_{2,1}^{3/2}$ with
	\begin{align}\label{1.20}
		\begin{split}
			\|a_0\|_{\dot B_{2,1}^{3/2}}+	\|\widetilde\theta_0\|_{\dot B_{2,1}^{3/2}}\leq \frac c{2M},
		\end{split}
	\end{align}
the system (\ref{idmodel}) has a unique global solution $(\rho,\theta)$ with initial data $\theta_0 = \widetilde\theta_0 + 1$ and $\rho_0 = 1/(1+a_0)$.
Moreover, if we define $\rho = 1/(1+a)$ and $\widetilde\theta = \theta - 1$, then for all $T>0$
	\begin{align}\label{1.21}
	\begin{split}
	\|(a,\widetilde\theta)\|_{E(T)}\leq c.
	\end{split}
\end{align}
	\end{thm}

  The rest of the paper unfolds as follows. Section 2 will present some basic tools in Fourier analysis: Littlewood-Paley decomposition and paraproduct calculus in Besov spaces. Section 3 will prove the global existence and uniqueness result in Soblolev spaces (Theorem \ref{t1.2}). Section 4 will prove the global well-posedness result in the critical Besov space by using Banach's fixed point Theorem.
\section{Notation and preliminaries}\label{sec_pre}

For any $1\le p\le \infty$ and measurable $f:\mathbb R^n \to \mathbb R$,
we will use $\|f\|_{L^p(\mathbb R^n)}$, $\|f\|_{L^p}$ or simply
$\|f\|_p$ to denote the usual $L^p$ norm. For a vector valued function
$f=(f^1,\cdots, f^m)$, we still denote $\|f\|_p :=\sum\limits_{j=1}^m \| f^j \|_p$.

For any $0<T<\infty$ and any Banach space $\mathbb B$ with norm $\| \cdot \|_{\mathbb B}$,
we will use the notation $C([0,T], \,\mathbb B)$ or $C_t^0 \mathbb B$
to denote the space of continuous
$\mathbb B$-valued functions endowed with the norm
\begin{align*}
\| f \|_{C([0,T], \mathbb B)} :=\max_{0\le t \le T} \| f(t) \|_{\mathbb B}.
\end{align*}
Also for $1\le p\le \infty$, we define
\begin{align*}
\| f \|_{L_t^p \mathbb B([0,T])} := \|  \| f(t) \|_{\mathbb B} \|_{L_t^p([0,T])}.
\end{align*}

We shall adopt the following convention for the Fourier transform:
\begin{align*}
& \hat f(\xi) = \int_{\mathbb R^n} f(x) e^{-i x\cdot \xi} dx;\\
& f(x)= \frac 1 {(2\pi)^n} \int_{\mathbb R^n} \hat f (\xi) e^{i x \cdot \xi} d\xi.
\end{align*}
For $s\in \mathbb R$, the fractional Laplacian $|\nabla|^s $ then corresponds to the Fourier
multiplier $|\xi|^s$ defined as
\begin{align*}
\widehat{|\nabla|^s f}(\xi) = |\xi|^s \hat f (\xi),
\end{align*}
whenever it is well-defined. For $s\ge 0$, $1\le p<\infty$, we define the semi-norm and norms:
\begin{align*}
&\| f \|_{\dot W^{s,p}} = \| |\nabla|^s f \|_p, \\
& \| f \|_{W^{s,p}} = \| |\nabla|^s f \|_p + \| f \|_p.
\end{align*}
When $p=2$ we denote $\dot H^s =\dot W^{s,2}$ and $H^s= W^{s,2}$ in accordance with
the usual notation.

For any two quantities $X$ and $Y$, we denote $X \lesssim Y$ if
$X \le C Y$ for some constant $C>0$. Similarly $X \gtrsim Y$ if $X
\ge CY$ for some $C>0$. We denote $X \sim Y$ if $X\lesssim Y$ and $Y
\lesssim X$. The dependence of the constant $C$ on
other parameters or constants are usually clear from the context and
we will often suppress  this dependence. We shall denote
$X \lesssim_{Z_1, Z_2,\cdots,Z_k} Y$
if $X \le CY$ and the constant $C$ depends on the quantities $Z_1,\cdots, Z_k$.

For any two quantities $X$ and $Y$, we shall denote $X\ll Y$ if
$X \le c Y$ for some sufficiently small constant $c$. The smallness of the constant $c$ is
usually clear from the context. The notation $X\gg Y$ is similarly defined. Note that
our use of $\ll$ and $\gg$ here is \emph{different} from the usual Vinogradov notation
in number theory or asymptotic analysis.

 We will need to use the Littlewood--Paley (LP) frequency projection
operators. To fix the notation, let $\phi_0$ be a radial function in
$C_c^\infty(\mathbb{R}^n )$ and satisfy
\begin{equation}\nonumber
0 \leq \phi_0 \leq 1,\quad \phi_0(\xi) = 1\ {\text{ for}}\ |\xi| \leq
1,\quad \phi_0(\xi) = 0\ {\text{ for}}\ |\xi| \geq 7/6.
\end{equation}
Let $\phi(\xi):= \phi_0(\xi) - \phi_0(2\xi)$ which is supported in $\frac 12 \le |\xi| \le \frac 76$.
For any $f \in \mathcal S(\mathbb R^n)$, $j \in \mathbb Z$, define
\begin{align*}
 &\widehat{S_{j} f} (\xi) = \phi_0(2^{-j} \xi) \hat f(\xi), \\
 &\widehat{\Delta_j f} (\xi) = \phi(2^{-j} \xi) \hat f(\xi), \qquad \xi \in \mathbb R^n.
\end{align*}
We will denote $P_{>j} = I-S_{j}$ ($I$ is the identity operator)
and for any $-\infty<a<b<\infty$, denote
$P_{[a,b]}=\sum_{a\leq j\leq b}\Delta_j$.  Sometimes for simplicity of
notation (and when there is no obvious confusion) we will write $f_j = \Delta_j f$ and
$f_{a\le\cdot\le b} = \sum_{j=a}^b f_j$.
By using the support property of $\phi$, we have $\Delta_j \Delta_{j^{\prime}} =0$ whenever $|j-j^{\prime}|>1$.

Thanks to the above Littlewood-Paley decomposition, a number of functional spaces can be characterized. Let us give the definition of
homogeneous Besov spaces at first.
\begin{defn}\label{D2.1}
For $s\in\R$, $(p,r)\in[1,\infty]^2$, and
$u\in\mathcal{S}'(\R^3),$ we set
$$\|u\|_{\dot B_{p,r}^s(\R^3)}=\left(\sum_{j\in\mathbb Z}2^{jsr}\|\Delta_ju\|_{L^p}^r\right)^{\frac{1}{r}},$$
with the usual modification if $r=\infty$.
\end{defn}
We then define the Besov space by
$\dot B_{p,r}^s=\{u\in\mathcal{S}'(\R^3),\
\|u\|_{\dot B_{p,r}^s(\R^3)}<\infty\}$. In the following, for convenience of notations, we always use $\dot B_{p,r}^s$ instead of $\dot B_{p,r}^s(\R^3)$ and similar notations for other norms. Let us now state some classical
properties for the Besov spaces without proof.
\begin{prop}\label{p2.1}
The following properties hold:

1) Derivatives: we have
$$\|\nabla u\|_{\dot B_{p,r}^{s-1}}\leq C\|u\|_{\dot B_{p,r}^{s}}.$$

2) Sobolev embedding: If $p_1\leq p_2 $  and $r_1\leq r_2,$ then
$\dot B_{p_1,r_1}^{s}\hookrightarrow
\dot B_{p_2,r_2}^{s-\frac{3}{p_1}+\frac{3}{p_2}}$.

\ \ \ If  $s_1>s_2$ and $1\leq p,r_1, r_2\leq +\infty,$ then
$\dot B_{p,r_1}^{s_1}\hookrightarrow \dot B_{p,r_2}^{s_2}.$

3) Algebraic property: for $s>0$, $\dot B_{p,r}^{s}\cap L^\infty$ is an
algebra.

4) Real interpolation:
$\left(\dot B_{p,r}^{s_1},\dot B_{p,r}^{s_2}\right)_{\theta,r'}=\dot B_{p,r'}^{\theta
s_1+(1-\theta)s_2}.$

\end{prop}

We recall some product laws in Besov spaces coming directly
 from the paradifferential calculus of J. M. Bony (see\cite{Bo}).
\begin{prop}\label{p2.2}
We have the following product laws:
$$\|uv\|_{\dot B_{p,r}^s}\lesssim\|u\|_{L^\infty}\|v\|_{\dot B_{p,r}^s}+\|v\|_{L^\infty}
\|v\|_{\dot B_{p,r}^s}\ \ \hbox{if}\ \ s>0,$$
$$\|uv\|_{\dot B_{p,r}^{s_1}}\lesssim\|u\|_{\dot B_{p,r}^{s_1}}
\|v\|_{\dot B_{p,r}^{s_2}}\ \ \hbox{if}\ \ s_1\leq\frac{3}{2},
 s_2>\frac{3}{2}
\ \ \hbox{and}\ \ s_1+s_2>0,$$
$$\|uv\|_{\dot B_{p,r}^{s_1+s_2-\frac{3}{2}}}\lesssim\|u\|_{\dot B_{p,r}^{s_1}}
\|v\|_{\dot B_{p,r}^{s_2}}\ \ \hbox{if}\ \ s_1,s_2<\frac{3}{2}
\ \ \hbox{and}\ \ s_1+s_2>0, $$
$$\|uv\|_{\dot B_{p,r}^{s}}\lesssim\|u\|_{\dot B_{p,r}^{s}}
\|v\|_{\dot B_{p,r}^{3/2}\cap L^\infty}\ \ \hbox{if}\ \
|s|<\frac{3}{2}.$$ Moreover, if $r=1$, the third inequality also holds for
$s_1,s_2\leq\frac{3}{2}$ and $s_1+s_2>0.$
\end{prop}

\section{Global well-posedness in Sobolev spaces}
The present section is dedicated to proving Theorem \ref{t1.2}. Before starting,  we
assume that $\widetilde\rho:=\rho-\underline\rho,\widetilde\theta:=\theta-\underline\theta$. Then  we first rewrite (\ref{idmodel}) as
\begin{equation}\label{4.3}
	\left \{
	\begin{aligned}
		&\partial_t\widetilde\rho-\kappa_1\bar\theta\Delta \widetilde \rho-\kappa_1\bar\rho\Delta \widetilde \theta=\kappa_1\Delta(\widetilde\rho\widetilde\theta),\\
		&\rho\kappa_2\partial_t\widetilde\theta-\kappa_1\kappa_2\nabla\widetilde\theta
		\cdot\nabla(\rho\theta)-\kappa_1^2\nabla\cdot(\theta\nabla(\rho\theta))=\nabla\cdot\left(\kappa_3(\theta)\nabla\widetilde\theta\right).\\
	\end{aligned} \right.
\end{equation}
For simplicity, here we assume that $\underline\rho=\underline\theta=1.$ And furthermore, we decompose the coefficients $\kappa_3(\theta)=\bar\kappa_3+\widetilde\kappa_3(\widetilde\theta)$, which satisfies $\widetilde\kappa_3(0)=0$. We also assume that $\widetilde\kappa_3'$ and
$\widetilde\kappa_3''$ exist and are bounded. Then  (\ref{4.3}) can be written by
\begin{equation}\label{4.4}
	\left \{
	\begin{aligned}
		&\partial_t\widetilde\rho-\kappa_1\Delta \widetilde \rho-\kappa_1\Delta \widetilde \theta=\kappa_1\Delta(\widetilde\rho\widetilde\theta),\\
		&\kappa_2\partial_t\widetilde\theta-(\kappa_1^2+\bar\kappa_3)\Delta\widetilde\theta-\kappa_1^2\Delta \widetilde \rho=\kappa_1(\kappa_1+\kappa_2)\Big(\nabla\widetilde\theta
		\cdot\nabla\widetilde\rho +\nabla\widetilde\theta
		\cdot\nabla\widetilde\theta\\
		&\quad\quad\quad\quad\quad\quad\quad\quad+\nabla\widetilde\theta
		\cdot\nabla(\widetilde\theta\widetilde\rho)\Big)+\kappa_1^2\Delta(\widetilde\rho\widetilde\theta)+\nabla\cdot(\widetilde\kappa_3(\widetilde\theta)\nabla\widetilde\theta)-\kappa_2\widetilde\rho\partial_t\widetilde\theta.\\
	\end{aligned} \right.
\end{equation}

\subsection{$L^2$ energy estimate}
Taking the $L^2$ inner product with $\widetilde\rho$ and $\widetilde \theta$ with respect to the first and second equations, one has
\begin{align}
	\begin{split}\label{4.5}
	\frac{1}{2}\frac{d}{dt}\|\widetilde\rho\|_{L^2}^2+\kappa_1\|\nabla \widetilde \rho\|^2_{L^2}&=-\kappa_1\int_{\RR^3}\nabla \widetilde\theta\cdot \nabla\widetilde\rho\,dx-\kappa_1\int_{\RR^3}\nabla (\widetilde\rho\widetilde\theta)\cdot \nabla\widetilde\rho\,dx\\
	&:=I_1+I_2.	
	\end{split}
\end{align}
and
\begin{align}
	\begin{split}\label{4.6}
		\frac{1}{2}\kappa_2&\frac{d}{dt}\|\widetilde\theta\|_{L^2}^2+(\kappa_1^2+\bar\kappa_3)\|\nabla \widetilde \theta\|^2_{L^2}\\
		&=-\kappa_1^2\int_{\RR^3}\nabla \widetilde\theta\cdot \nabla\widetilde\rho\,dx-\kappa_1^2\int_{\RR^3}\nabla (\widetilde\rho\widetilde\theta)\cdot \nabla\widetilde\theta\,dx\\
		&\quad +\kappa_1(\kappa_1+\kappa_2)\int_{\RR^3}\Big(\nabla \widetilde\rho\cdot\nabla\widetilde\theta+\nabla \widetilde\theta\cdot\nabla\widetilde\theta+\nabla\widetilde\theta\cdot\nabla(\widetilde\theta\widetilde\rho)\Big)\,\widetilde\theta\,dx\\
		&\quad-\int_{\RR^3}\widetilde\kappa_3(\widetilde\theta)\nabla\widetilde\theta\cdot\nabla\widetilde\theta\,dx+\kappa_2\int_{\RR^3}\widetilde\rho\partial_t \widetilde\theta \ \widetilde\theta \,dx\\
		&:=I_3+I_4+I_5+I_6+I_7,
	\end{split}
\end{align}
where
\begin{align*}
	&I_1=-\kappa_1\int_{\RR^3}\nabla \widetilde\theta\cdot \nabla\widetilde\rho\,dx,\quad I_2=-\kappa_1\int_{\RR^3}\nabla (\widetilde\rho\widetilde\theta)\cdot \nabla\widetilde\rho\,dx,\\
	& I_3=-\kappa_1^2\int_{\RR^3}\nabla \widetilde\theta\cdot \nabla\widetilde\rho\,dx,\quad I_4=-\kappa_1^2\int_{\RR^3}\nabla (\widetilde\rho\widetilde\theta)\cdot \nabla\widetilde\theta\,dx,\\
	&I_5=\kappa_1(\kappa_1+\kappa_2)\int_{\RR^3}\Big(\nabla \widetilde\rho\cdot\nabla\widetilde\theta+\nabla \widetilde\theta\cdot\nabla\widetilde\theta+\nabla\widetilde\theta\cdot\nabla(\widetilde\theta\widetilde\rho)\Big)\,\widetilde\theta\,dx,\\
	&I_6=-\int_{\RR^3}\widetilde\kappa_3(\widetilde\theta)\nabla\widetilde\theta\cdot\nabla\widetilde\theta\,dx,\quad I_7=\kappa_2\int_{\RR^3}\widetilde\rho\partial_t\widetilde\theta \ \widetilde\theta \,dx.
\end{align*}
Firstly, by H\"{o}lder and Cauchy inequalities, we have
$$I_1\leq \frac{1}{2}\kappa_1\|\nabla\widetilde\theta\|_{L^2}^2+\frac{1}{2}\kappa_1\|\nabla\widetilde\rho\|_{L^2}^2,\quad I_3\leq \frac{1}{2}\kappa^2_1\|\nabla\widetilde\theta\|_{L^2}^2+\frac{1}{2}\kappa^2_1\|\nabla\widetilde\rho\|_{L^2}^2.$$
Then by linear combination of (\ref{4.5}) and (\ref{4.6}), one can get that
\begin{align}
	\begin{split}\label{4.7}
		\frac{1}{2}\kappa_1(1+\delta)\frac{d}{dt}\|\widetilde\rho\|_{L^2}^2
		&+\frac{1}{2}\kappa_2\frac{d}{dt}\|\widetilde\theta\|_{L^2}^2+\frac12\delta\kappa_1^2\|\nabla \widetilde \rho\|^2_{L^2}+(\bar\kappa_3-\frac{1}{2}\delta\kappa_1^2)\|\nabla \widetilde \theta\|^2_{L^2}\\
		&\leq \kappa_1(1+\delta)I_2+I_4+I_5+I_6+I_7,
	\end{split}
\end{align}
where $\delta$ is a small  positive constant that satisfies
$$\delta\kappa_1^2<2\bar\kappa_3.$$
To bound $I_2$, we decompose it into two parts  and by H\"{o}lder inequality, we have
\begin{align}
	\begin{split}\label{4.8}
	I_2&=-\kappa_1\int_{\RR^3}\widetilde\theta \nabla \widetilde\rho \cdot \nabla\widetilde\rho\,dx-\kappa_1\int_{\RR^3}\widetilde\rho\nabla \widetilde\theta\cdot \nabla\widetilde\rho\,dx\\
	&\leq \kappa_1\|\widetilde\theta\|_{L^\infty}\|\nabla\widetilde\rho\|_{L^2}^2+\kappa_1\|\widetilde\rho \|_{L^\infty}\|\nabla\widetilde\rho\|_{L^2}\|\nabla\widetilde\theta\|_{L^2}.
	\end{split}
\end{align}
Similarly, $I_4$ and $I_5$ can be bounded by
\begin{align}
	\begin{split}\label{4.9}
		&I_4\leq \kappa^2_1\|\widetilde\rho\|_{L^\infty}\|\nabla\widetilde\theta\|_{L^2}^2+\kappa_1^2\|\widetilde\theta \|_{L^\infty}\|\nabla\widetilde\rho\|_{L^2}\|\nabla\widetilde\theta\|_{L^2},\\
		&I_5\leq \kappa_1(\kappa_1+\kappa_2)\Big(\|\widetilde\rho\|_{L^\infty}+\|\widetilde\theta\|_{L^\infty}+\|\widetilde\rho\|_{L^\infty}\|\widetilde\theta\|_{L^\infty}+\|\widetilde\theta\|^2_{L^\infty}\Big)\\
		&\quad\quad\times\Big(\|\nabla\widetilde\rho\|_{L^2}\|\nabla\widetilde\theta\|_{L^2}+\|\nabla\widetilde\theta\|_{L^2}^2\Big).
	\end{split}
\end{align}
For $I_6$,  notice that $\widetilde\kappa_3(0)=0$. Then we use Taylor formula and   H\"{o}lder inequality to get that
\begin{align}
	\begin{split}\label{4.10}
		I_6\leq C \|\widetilde\theta\|_{L^\infty}\|\nabla\widetilde\theta\|_{L^2}^2.
	\end{split}
\end{align}
Now we turn to the last term $I_7$. The H\"{o}lder inequality implies that
\begin{align}
	\begin{split}\label{4.11}
		I_7&=\kappa_2\int_{\RR^3}\widetilde\rho\ \partial_t\widetilde\theta\ \widetilde\theta  \,dx\leq \kappa_2 \|\widetilde\rho\|_{L^3}\|\partial_t\widetilde\theta\|_{L^2}\|\widetilde\theta\|_{L^6}\\
		&\leq C\|\rho\|_{H^1}\big(\|\partial_t\widetilde\theta\|_{L^2}^2+\|\nabla \widetilde\theta\|_{L^2}^2\big).
	\end{split}
\end{align}
Combining the estimates (\ref{4.8}) to (\ref{4.11}) with (\ref{4.7}), we have
\begin{align}
	\begin{split}\label{4.12}
		\frac{1}{2}\kappa_1(1+\delta)&\frac{d}{dt}\|\widetilde\rho\|_{L^2}^2
		+\frac{1}{2}\kappa_2\frac{d}{dt}\|\widetilde\theta\|_{L^2}^2+\frac{1}{2}\delta\kappa_1^2\|\nabla \widetilde \rho\|^2_{L^2}+(\bar\kappa_3-\frac{1}{2}\delta\kappa_1^2)\|\nabla \widetilde \theta\|^2_{L^2}\\
		&\leq C(\kappa_1^2+\kappa_2^2+1+\frac{\kappa_2}{\kappa_1})\Big(\|\widetilde\rho\|_{H^2}+\|\widetilde\theta\|_{H^2}+\|\widetilde\rho\|^2_{H^2}+\|\widetilde\theta\|^2_{H^2}\Big)\\
		&\quad\quad\times\Big(\|\partial_t\widetilde\theta\|_{L^2}^2+\|\nabla\widetilde\theta\|^2_{L^2}+\|\nabla\widetilde\rho\|_{H^2}^2\Big).
	\end{split}
\end{align}
\subsection{$\dot H^2$ energy estimate}
Due to the equivalence of $\|(\widetilde\rho,\widetilde\theta)\|_{H^2}$ with $\|(\widetilde\rho,\widetilde\theta)\|_{L^2}+\|(\widetilde\rho,\widetilde\theta)\|_{\dot H^2}$, it is sufficient to bound the homogeneous $\dot H^2$ norm of $(\widetilde\rho,\widetilde\theta)$.  Applying $\partial_i^2$ for $i=1,2,3$ to (
\ref{4.4})  and then taking the $L^2 $ inner product with $(\partial_i^2\widetilde\rho,\partial_i^2\widetilde\theta)$, respectively, we find that
\begin{align}
	\begin{split}\label{4.13}
		&\frac{1}{2}\frac{d}{dt}\|\partial_i^2\widetilde\rho\|_{L^2}^2+\kappa_1\|\nabla \widetilde \partial_i^2\rho\|^2_{L^2}\\
		&=-\kappa_1\int_{\RR^3}\nabla\partial_i^2 \widetilde\theta\cdot \nabla\partial_i^2\widetilde\rho\,dx-\kappa_1\int_{\RR^3}\nabla\partial_i^2 (\widetilde\rho\widetilde\theta)\cdot \nabla\partial_i^2\widetilde\rho\,dx\\
		&\leq \frac{1}{2}\kappa_1\|\nabla\partial_i^2\widetilde\theta\|_{L^2}^2+\frac{1}{2}\kappa_1\|\nabla\partial_i^2\widetilde\rho\|_{L^2}^2-\kappa_1\int_{\RR^3}\nabla\partial_i^2 (\widetilde\rho\widetilde\theta)\cdot \nabla\partial_i^2\widetilde\rho\,dx,
	\end{split}
\end{align}
and
\begin{align}
	\begin{split}\label{4.14}
		\frac{1}{2}\kappa_2&\frac{d}{dt}\|\partial_i^2\widetilde\theta\|_{L^2}^2+(\kappa_1^2+\bar\kappa_3)\|\nabla \partial_i^2\widetilde \theta\|^2_{L^2}\\
		&=-k_1^2\int_{\RR^3}\nabla \partial_i^2\widetilde\theta\cdot \partial_i^2\nabla\widetilde\rho\,dx-k_1^2\int_{\RR^3}\nabla\partial_i^2 (\widetilde\rho\widetilde\theta)\cdot \nabla\partial_i^2\widetilde\theta\,dx\\
		&\quad +\kappa_1(\kappa_1+\kappa_2)\int_{\RR^3}\partial_i^2\Big(\nabla \widetilde\rho\cdot\nabla\widetilde\theta+\nabla \widetilde\theta\cdot\nabla\widetilde\theta+\nabla\widetilde\theta\cdot\nabla(\widetilde\theta\widetilde\rho)\Big)\,\partial_i^2\widetilde\theta\,dx\\
		&\quad-\int_{\RR^3}\partial_i^2(\widetilde\kappa_3(\widetilde\theta)\nabla\widetilde\theta)\cdot\nabla\partial_i^2\widetilde\theta\,dx+\kappa_2\int_{\RR^3}\partial_i^2(\widetilde\rho\partial_t\widetilde\theta) \ \partial_i^2\widetilde\theta \,dx\\
		&\leq \frac{1}{2}\kappa^2_1\|\nabla\partial_i^2\widetilde\theta\|_{L^2}^2+\frac{1}{2}\kappa^2_1\|\nabla\partial_i^2\widetilde\rho\|_{L^2}^2-k_1^2\int_{\RR^3}\nabla\partial_i^2 (\widetilde\rho\widetilde\theta)\cdot \nabla\partial_i^2\widetilde\theta\,dx\\
		&\quad +\kappa_1(\kappa_1+\kappa_2)\int_{\RR^3}\partial_i^2\Big(\nabla \widetilde\rho\cdot\nabla\widetilde\theta+\nabla \widetilde\theta\cdot\nabla\widetilde\theta+\nabla\widetilde\theta\cdot\nabla(\widetilde\theta\widetilde\rho)\Big)\,\partial_i^2\widetilde\theta\,dx\\
		&\quad-\int_{\RR^3}\partial_i^2(\widetilde\kappa_3(\widetilde\theta)\nabla\widetilde\theta)\cdot\nabla\partial_i^2\widetilde\theta\,dx+\kappa_2\int_{\RR^3}\partial_i^2(\widetilde\rho\partial_t\widetilde\theta) \ \partial_i^2\widetilde\theta \,dx.
	\end{split}
\end{align}
Again, we denote $J_1$ to $J_5$ by
\begin{align*}
	&J_1=-\kappa_1\int_{\RR^3}\nabla\partial_i^2 (\widetilde\rho\widetilde\theta)\cdot \nabla\partial_i^2\widetilde\rho\,dx, \quad J_2=-k_1^2\int_{\RR^3}\nabla \partial_i^2(\widetilde\rho\widetilde\theta)\cdot \nabla\partial_i^2\widetilde\theta\,dx,\\
	&J_3=\kappa_1(\kappa_1+\kappa_2)\int_{\RR^3}\partial_i^2\Big(\nabla \widetilde\rho\cdot\nabla\widetilde\theta+\nabla \widetilde\theta\cdot\nabla\widetilde\theta+\nabla\widetilde\theta\cdot\nabla(\widetilde\theta\widetilde\rho)\Big)\,\partial_i^2\widetilde\theta\,dx,\\
	&J_4=-\int_{\RR^3}\partial_i^2(\widetilde\kappa_3(\widetilde\theta)\nabla\widetilde\theta)\cdot\nabla\partial_i^2\widetilde\theta\,dx,\quad J_5=\kappa_2\int_{\RR^3}\partial_i^2(\widetilde\rho\partial_t\widetilde\theta) \ \partial_i^2\widetilde\theta \,dx.
\end{align*}
Then the linear combination of (\ref{4.13}) and (\ref{4.14}) implies that
\begin{align}
	\begin{split}\label{4.15}
		\frac{1}{2}\kappa_1(1+\delta)\frac{d}{dt}\|\partial_i^2\widetilde\rho\|_{L^2}^2
		&+\frac{1}{2}\kappa_2\frac{d}{dt}\|\partial_i^2\widetilde\theta\|_{L^2}^2\\
		&+\frac{1}{2}\delta\kappa_1^2\|\nabla\partial_i^2 \widetilde \rho\|^2_{L^2}+(\bar\kappa_3-\frac{1}{2}\delta\kappa_1^2)\|\nabla\partial_i^2 \widetilde \theta\|^2_{L^2}\\
		&\leq \kappa_1(1+\delta)J_1+J_2+J_3+J_4+J_5.
	\end{split}
\end{align}
By the H\"{o}lder inequality, one has
\begin{align}
	\begin{split}\label{4.16}
		J_1&\leq C\kappa_1\Big(\|\widetilde\theta\|_{L^\infty}\|\nabla\partial_i^2\widetilde\rho\|_{L^2}^2+\|\nabla\partial_i\widetilde\rho\|_{L^6}\|\partial_i\widetilde\theta\|_{L^3}\|\nabla\partial_i^2\widetilde\rho\|_{L^2}\\
		&\quad+\|\nabla\widetilde\rho\|_{L^3}\|\partial_i^2\widetilde\theta\|_{L^6}\|\nabla\partial_i^2\widetilde\rho\|_{L^2}+\|\widetilde\rho\|_{L^\infty}\|\nabla\partial_i^2\widetilde\rho\|_{L^2}\|\nabla\partial_i^2\widetilde\theta\|_{L^2}\\
		&\quad+\|\nabla\partial_i\widetilde\theta\|_{L^6}\|\partial_i\widetilde\rho\|_{L^3}\|\nabla\partial_i^2\widetilde\rho\|_{L^2}+\|\nabla\widetilde\theta\|_{L^3}\|\partial_i^2\widetilde\rho\|_{L^6}\|\nabla\partial_i^2\widetilde\rho\|_{L^2}\Big)\\
		&\leq C\kappa_1\Big(\|\widetilde\theta\|_{H^2}\|\nabla\partial_i^2\widetilde\rho\|_{L^2}^2+\|\widetilde\rho\|_{H^2}\|\nabla\partial_i^2\widetilde\rho\|_{L^2}\|\nabla\partial_i^2\widetilde\theta\|_{L^2}\Big).
		\end{split}
\end{align}
Similarly, we have
\begin{align}
	\begin{split}\label{4.17}
		J_2\leq  C\kappa_1^2\Big(\|\widetilde\theta\|_{H^2}\|\nabla\partial_i^2\widetilde\rho\|_{L^2}\|\nabla\partial_i^2\widetilde\theta\|_{L^2}+\|\widetilde\rho\|_{H^2}\|\nabla\partial_i^2\widetilde\theta\|_{L^2}^2\Big).
	\end{split}
\end{align}
and
\begin{align}
	\begin{split}\label{4.18}
	J_3&\leq C(\kappa_1^2+\kappa_2^2)\|\nabla\partial_i^2\widetilde\theta\|_{L^2}\Big(\|\nabla\widetilde\theta\|_{L^3}\|\nabla\partial_i\widetilde\rho\|_{L^6}+\|\nabla\widetilde\rho\|_{L^3}\|\nabla\partial_i\widetilde\theta\|_{L^6}\\
	&\quad+\|\nabla\widetilde\theta\|_{L^3}\|\nabla\partial_i\widetilde\theta\|_{L^6}+\|\widetilde\rho\|_{L^\infty}\|\nabla\widetilde\theta\|_{L^3}\|\nabla\partial_i\widetilde\theta\|_{L^6}\\
	&\quad+\|\widetilde\theta\|_{L^\infty}\|\nabla\widetilde\rho\|_{L^3}\|\nabla\partial_i\widetilde\theta\|_{L^6}+\|\widetilde\theta\|_{L^\infty}\|\nabla\widetilde\theta\|_{L^3}\|\nabla\partial_i\widetilde\rho\|_{L^6}\\
	&\quad+{\|\nabla\widetilde\rho\|_{L^\infty}\|\nabla\widetilde\theta\|_{L^3}\|\nabla\widetilde\theta\|_{L^6}}\Big)\\
	&\leq C(\kappa_1^2+\kappa_2^2)\|\nabla\partial_i^2\widetilde\theta\|_{L^2}\Big(\|\widetilde\theta\|_{H^2}+\|\widetilde\rho\|_{H^2}+\|\widetilde\theta\|^2_{H^2}+\|\widetilde\rho\|^2_{H^2}\Big)\\
	&\quad\times\Big(\|\nabla\partial_i^2\widetilde\rho\|_{L^2}
+\|\nabla\partial_i^2\widetilde\theta\|_{L^2}+\|\nabla^2\widetilde\theta\|_{L^2}\Big),
\end{split}
\end{align}
where we have used the Sobolev embedding that for any distribution $f$,
$$\|f\|_{L^\infty}\leq C\|f\|_{H^2},\quad \|f\|_{L^3}\leq C\|\nabla f\|_{H^1}.$$
Similarly, for $J_4$, we have
\begin{align}
	\begin{split}\label{4.19}
		J_4&\leq C \|\widetilde\theta\|_{L^\infty}\|\nabla\partial_i^2\widetilde\theta\|_{L^2}^2+\|\partial_i\widetilde\theta\|_{L^3}\|\nabla\partial_i\widetilde\theta\|_{L^6}^2\|\nabla\partial_i^2\widetilde\theta\|_{L^2}\\
		&\quad+\|\partial_i\widetilde\theta\|_{L^\infty}\|\partial_i\widetilde\theta\|_{L^3}^2\|\nabla\partial_i^2\widetilde\theta\|_{L^2}\\
		&\leq C\|\widetilde\theta\|_{H^2} \|\nabla\partial_i^2\widetilde\theta\|_{L^2}^2+C\|\widetilde\theta\|_{H^2}\|\nabla\partial_i\widetilde\theta\|_{L^2}\|\nabla\partial_i^2\widetilde\theta\|_{L^2}.
	\end{split}
\end{align}
As to the last term $J_5$, we decompose it into four parts
\begin{align*}
	\begin{split}
		J_5&=\kappa_2\int_{\RR^3}\partial_i^2(\widetilde\rho\partial_t\widetilde\theta) \ \partial_i^2\widetilde\theta \,dx\\
		&=\kappa_2\int_{\RR^3}\partial_i^2\widetilde\rho\ \partial_t\widetilde\theta \ \partial_i^2\widetilde\theta \,dx+2\kappa_2\int_{\RR^3}\partial_i\widetilde\rho
	\ \partial_t\partial_i\widetilde\theta \ \partial_i^2\widetilde\theta \,dx\\
	&\quad-\frac{1}{2}\kappa_2\int_{\RR^3}\partial_t\widetilde\rho\ (\partial_i^2\widetilde\theta)^2 \,dx+\frac{1}{2}\kappa_2\int_{\RR^3}\partial_t\big(\widetilde\rho\ (\partial_i^2\widetilde\theta)^2 \big)\,dx\\
	&:=J_{51}+J_{52}+J_{53}+J_{54}.
	\end{split}
\end{align*}
For $J_{53}$,  we can quickly by the equation of $\widetilde\rho$ get that
\begin{align}
	\begin{split}\label{4.20}
		J_{53}&=-\frac{1}{2}\kappa_2\int_{\RR^3}\partial_t\widetilde\rho\ (\partial_i^2\widetilde\theta)^2  \,dx\\
		&=- \frac{1}{2} \kappa_1 \kappa_2\int_{\RR^3}\Big(\Delta \widetilde\rho+\Delta \widetilde\theta+\Delta (\widetilde\rho\widetilde\theta)\Big)\ (\partial_i^2\widetilde\theta)^2 \,dx\\
		&\leq C\kappa_1 \kappa_2 \Big(\|\widetilde\rho\|_{H^2}+\|\widetilde\theta\|_{H^2}+\|\widetilde\rho\|^2_{H^2}+\|\widetilde\theta\|^2_{H^2}\Big)\Big(\|\nabla\widetilde\theta\|^2_{H^2}+\|\nabla\widetilde\rho\|_{H^2}^2\Big).
	\end{split}
\end{align}
To bound $J_{51}$ and $J_{52}$,
we roughly estimate them as follows,
\begin{align}
	\begin{split}\label{4.20a}
		J_{51}+J_{52}\leq C(\kappa_1,\kappa_2)\Big(\|\partial_i^2\widetilde\rho\ \partial_i^2\widetilde\theta \|_{L^2}^2+\|\partial_i\widetilde\rho\ \partial_i^2\widetilde\theta \|_{L^2}^2\Big)+\frac{\kappa_2^2}{16C\kappa_1^2}\delta\|\partial_t\partial_i^k\widetilde\theta\|_{L^2}^2.
	\end{split}
\end{align}
 Let us postpone the estimate of $J_{54}$, since its bound can be deduced directly by taking time integral. We need to get the additional estimates on $\partial_t\widetilde\theta$ and  $\partial_t\partial_i\widetilde\theta$. For that, applying $\partial_i^k$ with $k=0,1,i=1,2,3$ to the equation of  $\widetilde\theta$, taking $L^2$ inner product with $\partial_t\partial_i^k\widetilde\theta$, one has
\begin{align}
	\begin{split}\label{4.21}
		&\kappa_2\|\partial_t\partial_i^k\widetilde\theta\|_{L^2}^2+\frac{1}{2}(\kappa_1^2+\bar\kappa_3)\frac{d}{dt}\|\nabla \partial_i^k\widetilde \theta\|^2_{L^2}\\
		&=k_1^2\int_{\RR^3} \Delta\partial_i^k\widetilde\rho\cdot\partial_t\partial_i^k\widetilde\theta\,dx+k_1^2\int_{\RR^3}\Delta\partial_i^k(\widetilde\rho\widetilde\theta)\cdot \partial_t\partial_i^k\widetilde\theta\,dx\\
		&\quad +\kappa_1(\kappa_1+\kappa_2)\int_{\RR^3}\partial_i^k\Big(\nabla \widetilde\rho\cdot\nabla\widetilde\theta+\nabla \widetilde\theta\cdot\nabla\widetilde\theta+\nabla\widetilde\theta\cdot\nabla(\widetilde\theta\widetilde\rho)\Big)\,\partial_t\partial_i^k\widetilde\theta\,dx\\
		&\quad+\int_{\RR^3}\partial_i^k\nabla\cdot(\widetilde\kappa_3(\widetilde\theta)\nabla\widetilde\theta)\cdot\partial_t\partial_i^k\widetilde\theta\,dx+\kappa_2\int_{\RR^3}\partial_i^k(\widetilde\rho\partial_t\widetilde\theta) \cdot \partial_t\partial_i^k\widetilde\theta \,dx\\
		&\leq C\frac{\kappa^4_1}{\kappa_2}\|\Delta\partial_i^k\widetilde\rho\|_{L^2}^2+C\frac{\kappa^4_1}{\kappa_2}\|\Delta\partial_i^k (\widetilde\rho\widetilde\theta)\|_{L^2}+C\frac{\kappa^2_1(\kappa_1+\kappa_2)^2}{\kappa_2}\|\partial_i^k\big(\nabla \widetilde\rho\cdot\nabla\widetilde\theta\big)\|_{L^2}^2\\
		&\quad+C\frac{\kappa^2_1(\kappa_1+\kappa_2)^2}{\kappa_2}\|\partial_i^k\big(\nabla\widetilde\theta\cdot\nabla\widetilde\theta\big)\|_{L^2}^2+C\frac{\kappa^2_1(\kappa_1+\kappa_2)^2}{\kappa_2}\|\partial_i^k\big(\nabla\widetilde\theta\cdot\nabla(\widetilde\theta\widetilde\rho)\big)\|_{L^2}^2\\
		&\quad+C\frac{1}{\kappa_2}\|\partial_i^k\nabla\big(\widetilde\kappa_3(\widetilde\theta)\nabla\widetilde\theta\big)\|^2_{L^2}+\kappa_2\|\partial_i^k\big(\widetilde\rho\partial_t\widetilde\theta\big)\|_{L^2} \|\partial_t\partial_i^k\widetilde\theta\|_{L^2}+\frac{1}{2}\kappa_2\|\partial_t\partial_i^k\widetilde\theta\|^2_{L^2},
	\end{split}
\end{align}
which implies
\begin{align}
	\begin{split}\label{4.22}
		&\frac{1}{2}\kappa_2\|\partial_t\partial_i^k\widetilde\theta\|_{L^2}^2+\frac{1}{2}(\kappa_1^2+\bar\kappa_3)\frac{d}{dt}\|\nabla \partial_i^k\widetilde \theta\|^2_{L^2}\\
		&\leq C\frac{\kappa^4_1}{\kappa_2}\|\Delta\partial_i^k\widetilde\rho\|_{L^2}^2+C\frac{\kappa^4_1}{\kappa_2}\|\Delta\partial_i^k (\widetilde\rho\widetilde\theta)\|^2_{L^2}+C\frac{\kappa^2_1(\kappa_1+\kappa_2)^2}{\kappa_2}\|\partial_i^k\big(\nabla \widetilde\rho\cdot\nabla\widetilde\theta\big)\|_{L^2}^2\\
		&\quad+C\frac{\kappa^2_1(\kappa_1+\kappa_2)^2}{\kappa_2}\|\partial_i^k\big(\nabla\widetilde\theta\cdot\nabla\widetilde\theta\big)\|_{L^2}^2+C\frac{\kappa^2_1(\kappa_1+\kappa_2)^2}{\kappa_2}\|\partial_i^k\big(\nabla\widetilde\theta\cdot\nabla(\widetilde\theta\widetilde\rho)\big)\|_{L^2}^2\\
		&\quad+C\frac{1}{\kappa_2}\|\partial_i^k\nabla\big(\widetilde\kappa_3(\widetilde\theta)\nabla\widetilde\theta\big)\|^2_{L^2}+\kappa_2\|\partial_i^k\big(\widetilde\rho\partial_t\widetilde\theta\big)\|_{L^2} \|\partial_t\partial_i^k\widetilde\theta\|_{L^2}.
	\end{split}
\end{align}
Multiplying by $\frac{\kappa_2}{4C\kappa_1^2}\delta$ on both sides of (\ref{4.22}) and combining the resulting inequality with (\ref{4.15})-(\ref{4.20}), we get that
\begin{align}
	\begin{split}\label{4.23}
		&\frac{1}{2}\kappa_1(1+\delta)\frac{d}{dt}\|\partial_i^2\widetilde\rho\|_{L^2}^2
	+\frac{1}{2}\kappa_2\frac{d}{dt}\|\partial_i^2\widetilde\theta\|_{L^2}^2+\frac{\kappa_2}{8C}(1+\frac{\bar\kappa_3}{\kappa_1^2})\frac{d}{dt}\|\nabla \partial_i^k\widetilde \theta\|^2_{L^2}\\
		&\quad+\frac{\delta\kappa_1^2}{4}\|\nabla\partial_i^2 \widetilde \rho\|^2_{L^2}+(\bar\kappa_3-\frac{1}{2}\delta\kappa_1^2)\|\nabla\partial_i^2 \widetilde \theta\|^2_{L^2}+\frac{\kappa_2^2}{16C\kappa_1^2}\delta\|\partial_t\partial_i^k\widetilde\theta\|_{L^2}^2\\
		&\leq C(\kappa_1,\kappa_2,\bar\kappa_3)\Big(\|\widetilde\theta\|_{H^2}+\|\widetilde\rho\|_{H^2}+\|\widetilde\theta\|^2_{H^2}+\|\widetilde\rho\|^2_{H^2}\Big)\\
		&\quad\times\Big(\|\nabla\widetilde\rho\|^2_{H^2}+\|\nabla\widetilde\theta\|^2_{H^2}+\|\partial_t\partial_i^k\widetilde\theta\|^2_{L^2}\Big)+J_{54}.
	\end{split}
\end{align}
Integrating (\ref{4.23}) over $[0,t]$ with respect to time variable, denoting
$$X(t)=\sup_{0\le\tau\le t}\left(\|\widetilde\rho(\tau)\|^2_{H^2}+	\|\widetilde\theta(\tau)\|^2_{H^2}\right)+\int_0^t	\Big(\|\nabla \widetilde\rho\|^2_{H^2}+	\|\nabla \widetilde\theta\|^2_{H^2}+\|\partial_t\widetilde\theta(t)\|^2_{H^1}\Big)d\tau,$$
$$X(0)=\|\widetilde\rho_0\|^2_{H^2}+	\|\widetilde\theta_0\|^2_{H^2},$$
combining with
$$\int_0^tJ_{54}\, d\tau \leq C(\kappa_2)(\|\widetilde\rho\|_{H^2}\|\widetilde \theta\|^2_{H^2}-\|\widetilde\rho_0\|_{H^2}\|\widetilde \theta_0\|^2_{H^2}),$$
we can obtain that
\begin{align}
	\begin{split}\nonumber
	X(t)\leq C_1(\kappa_1,\kappa_2,\bar\kappa_3)\left(X(0)+(X(0))^{\frac{3}{2}}\right)+C_2(\kappa_1,\kappa_2,\bar\kappa_3)\left(X(t)^{\frac{3}{2}}+X(t)^{2}\right),
	\end{split}
\end{align}
By using the standard continuity argument, we complete the proof of Theorem \ref{t1.2}.

\section{Global existence for small data with critical regularity}
In this section, we will  obtain the global existence of solutions to system (\ref{idmodel}) in Theorem \ref{t1.3}. From now on, we define the density and the temperature by the form
$$a:=\frac{1}{\rho}-1,\quad \widetilde\theta:=\theta-1.$$
Then system (\ref{idmodel}) can be rewritten as
\begin{equation}\label{5.2}
	\left \{
	\begin{aligned}
		&\partial_ta-\kappa_1\Delta a+\kappa_1\Delta \widetilde \theta=F,\\
		&\kappa_2\partial_t\widetilde\theta-(\kappa_1^2+\bar\kappa_3)\Delta\widetilde\theta+\kappa_1^2\Delta a=G,
	\end{aligned} \right.
\end{equation}
where
\begin{align}
	\begin{split}\label{5.3}
		&F=-2\kappa_1\frac{\widetilde\theta + 1}{1+a}|\nabla a |^2 +2\kappa_1 \nabla a\cdot \nabla  \widetilde\theta-\kappa_1a\Delta\widetilde\theta+\kappa_1 \widetilde\theta \Delta a,\\
		&G=2 \kappa_1^2 \frac{(\widetilde\theta+1)^2}{(1+a)^2}|\nabla a|^2 - (3\kappa_1^2 + \kappa_1 \kappa_2) \frac{(\widetilde\theta+1)}{1+a} \nabla a\cdot \nabla  \widetilde\theta		\\
		&\quad -\kappa_1^2 \frac{\widetilde\theta^2 + 2\widetilde\theta}{1+a}\Delta a + \kappa_1^2 \frac{a}{1+a} \Delta a
	+\bar\kappa_3 a \Delta\widetilde\theta+\kappa_1^2\widetilde \theta\Delta \widetilde \theta\\
		&\quad +(1+a)\nabla\cdot(\widetilde\kappa_3(\widetilde\theta)\nabla\widetilde\theta) + \kappa_1(\kappa_1+\kappa_2)|\nabla\widetilde\theta|^2.
	\end{split}
\end{align}

Proving the global existence result is based on the following variant of Banach's fixed point theorem. For the proof, we refer e.g. to \cite{KP}.
\begin{lem}\label{KP}
Let $X$ be a reflexive Banach space or let $X$ have a separable pre-dual. Let $K$ be a convex, closed and bounded subset of $X$ and assume that $X$ is embedded into a Banach space Y. Let $\Phi$ : $X\longrightarrow X$  map $K$ into $K$ and assume there exists $c < 1$ such that
$$\|\Phi(x)-\Phi(y)\|_Y\leq c\|x-y\|_Y, \quad x,y\in K.$$
Then there exists a unique fixed point of $\Phi$ in $K$.
\end{lem}
Based on the natural scaling of system (\ref{idmodel}), we choose our working space to be
$$E(T):=\left\{u\in \mathcal C\left([0,T],\dot B_{2,1}^{3/2}\right),\quad \nabla^2u\in L^1\left(0,T;\dot B_{2,1}^{3/2}\right)\right\},\quad T>0,$$
with the norm
$$\|u\|_{E(T)}:=\|u\|_{L^{\infty}_T(\dot B_{2,1}^{3/2})}+\|\nabla^2u\|_{L^{1}_T(\dot B_{2,1}^{3/2})}.$$
The following proposition quantifies the smoothing effect of the linear system of \eqref{5.2}.
\begin{prop}\label{p5.2}
	Let us consider the initial data $(a_0,\widetilde\theta_0)$ in $\dot B_{2,1}^{s}(\RR^3)$ with regularity $s\leq \frac{3}{2}$. Introducing a pair of forces $(F,G)$ in $L^1_t(\dot B_{2,1}^{s}(\RR^3))$, we denote by $(a,\widetilde\theta)$ the unique solution of the following linear parabolic system:
	\begin{equation}\label{5.4}
		\left \{
		\begin{aligned}
			&\partial_ta-\kappa_1\Delta a+\kappa_1\Delta \widetilde \theta=F,\\
			&\kappa_2\partial_t\widetilde\theta-(\kappa_1^2+\bar\kappa_3)\Delta\widetilde\theta+\kappa_1^2\Delta a=G.
		\end{aligned} \right.
	\end{equation}
Then $(a,\widetilde\theta)$ belongs to $L^\infty_t(\dot B_{2,1}^{s}(\RR^3))$ and the pairs $(\partial_ta,\partial_t\widetilde\theta)$ and $(\Delta a,\Delta\widetilde\theta)$ belong to  $L^1_t(\dot B_{2,1}^{s}(\RR^3))$. Furthermore, there exists a positive constant $C$ depending only on $\kappa_1,\kappa_2$ and $\bar\kappa_3$ such that
\begin{align}
	\begin{split}\label{5.5}
		\|(a,\widetilde\theta)\|_{L^\infty_t(\dot B_{2,1}^{s})}&+\|(\partial_ta,\partial_t\widetilde\theta)\|_{L^1_t(\dot B_{2,1}^{s})}+\|(\Delta a,\Delta\widetilde\theta)\|_{L^1_t(\dot B_{2,1}^{s})}\\
		&\leq C(\kappa_1,\kappa_2,\bar\kappa_3)	\Big( \|(a_0,\widetilde\theta_0)\|_{L^\infty_t(\dot B_{2,1}^{s})}+\|(F,G)\|_{L^1_t(\dot B_{2,1}^{s})}\Big).
	\end{split}
\end{align}
\end{prop}

\begin{proof}
	We first apply the homogeneous dyadic block $\dot\Delta_q$ to system (\ref{5.4}), and multiply both equations by $\dot\Delta_qa$ and $\dot\Delta_q\widetilde\theta$, respectively, and integrate over $\RR^3$. We then get
	\begin{align}
		\begin{split}\label{5.6}
			\frac{1}{2}\frac{d}{dt}\|\dot\Delta_qa\|_{L^2}^2+\kappa_1\|\nabla\dot\Delta_qa\|^2_{L^2}&=\kappa_1\int_{\RR^3}\nabla \dot\Delta_q\widetilde\theta\cdot \nabla\dot\Delta_qa\,dx+\int_{\RR^3}\dot\Delta_qF\  \dot\Delta_qa\,dx,\\
		\end{split}
	\end{align}
	and
	\begin{align}
		\begin{split}\label{5.7}
			\frac{1}{2}\kappa_2\frac{d}{dt}\|\dot\Delta_q\widetilde\theta\|_{L^2}^2&+(\kappa_1^2+\bar\kappa_3)\|\nabla \dot\Delta_q\widetilde \theta\|^2_{L^2}\\
			&=\kappa_1^2\int_{\RR^3}\nabla \dot\Delta_q\widetilde\theta\cdot \nabla\dot\Delta_q a\,dx+\int_{\RR^3}\dot\Delta_qG \ \dot\Delta_q\widetilde\theta \,dx.
		\end{split}
	\end{align}
	By H\"{o}lder and Cauchy inequalities, the first term on the right side of (\ref{5.6}) can be bounded by
	$$\kappa_1\int_{\RR^3}\nabla \dot\Delta_q\widetilde\theta\cdot \nabla\dot\Delta_qa\,dx\leq \frac{1}{2}\kappa_1\|\nabla \dot\Delta_q\widetilde\theta\|^2_{L^2}+ \frac{1}{2}\kappa_1\|\nabla \dot\Delta_qa\|^2_{L^2}.$$
	Plugging the above inequality into (\ref{5.6}), and making the linear combination of the resulting inequality with (\ref{5.7}), one has
	\begin{align}
		\begin{split}\label{5.8}
			\frac{1}{2}\kappa_1(1+\delta)\frac{d}{dt}\|\dot\Delta_qa\|_{L^2}^2
			&+\frac{1}{2}\kappa_2\frac{d}{dt}\|\dot\Delta_q\widetilde\theta\|_{L^2}^2+\delta\kappa_1^2\|\nabla \dot\Delta_qa\|^2_{L^2}+(\bar\kappa_3-\frac{1}{2}\delta\kappa_1^2)\|\nabla \dot\Delta_q\widetilde\theta\|^2_{L^2}\\
			&\leq \kappa_1(1+\delta)\|\dot\Delta_qF\|_{L^2}\|\dot\Delta_qa\|_{L^2}+\|\dot\Delta_qG\|_{L^2}\|\dot\Delta_q\widetilde\theta\|_{L^2},
		\end{split}
	\end{align}
	where $\delta$ is a small positive number. Setting
	$$f_q^2=\kappa_1(1+\delta)\|\dot\Delta_qa\|_{L^2}^2
	+\kappa_2\|\dot\Delta_q\widetilde\theta\|_{L^2}^2$$
	and $\kappa=\min\{\frac{\delta\kappa_1}{1+\delta},\frac{\bar\kappa_3-\frac{1}{2}\delta\kappa_1^2}{\kappa_2}\}$, we then by Bernstein inequality have
	\begin{align}
		\begin{split}\label{5.9}
			\frac{1}{2}\frac{d}{dt}f_q^2+\kappa2^{2q}f_q^2\leq C_{\kappa_1,\kappa_2}(\|\dot\Delta_qF\|_{L^2}+\|\dot\Delta_qG\|_{L^2})f_q.
		\end{split}
	\end{align}
	To finish this, we multiply the above inequality by $2^{2qs}$, and denote
	$$g_q=2^{qs}\sqrt{\kappa_1(1+\delta)\|\dot\Delta_qa\|_{L^2}^2
		+\kappa_2\|\dot\Delta_q\widetilde\theta\|_{L^2}^2},$$
	we then get 	
	\begin{align}
		\begin{split}\label{5.10}
			\frac{1}{2}\frac{d}{dt}g_q^2+\kappa2^{2q}g_q^2\leq C(\kappa_1,\kappa_2)2^{qs}(\|\dot\Delta_qF\|_{L^2}+\|\dot\Delta_qG\|_{L^2})g_q.
		\end{split}
	\end{align}
	Using $h_q^2=g_q^2+\epsilon^2$, integrating over $[0,t]$ and then letting $\epsilon$ tend to 0, we infer
	\begin{align}
		\begin{split}\label{5.11}
			g_q(t)+\kappa2^{2q}\int_0^tg_q(\tau)\,d\tau\leq g_q(0)+ C(\kappa_1,\kappa_2)2^{qs}\int_0^t(\|\dot\Delta_qF\|_{L^2}+\|\dot\Delta_qG\|_{L^2})d\,\tau.
		\end{split}
	\end{align}
	We finally conclude that
	\begin{align}
		\begin{split}\label{5.12}
			\|(a,\widetilde\theta)\|_{L^\infty_t(\dot B_{2,1}^{s})}&+\|( a,\widetilde\theta)\|_{L^1_t(\dot B_{2,1}^{s+2})}\\
			&\leq C(\kappa_1,\kappa_2)	\Big( \|(a_0,\widetilde\theta_0)\|_{L^\infty_t(\dot B_{2,1}^{s})}+\|(F,G)\|_{L^1_t(\dot B_{2,1}^{s})}\Big).
		\end{split}
	\end{align}
	Combining (\ref{5.12}) with the equations of $(a,\widetilde\theta)$, we eventually finished the proof of this proposition.
\end{proof}
Our construction of the global solution relies on a combination of Proposition \ref{p5.2} with Lemma \ref{KP}. To this end, for any given $T>0$, we define the set $K(T)$ by
$$K(T):=\{(b,\tau)\in E(T)\times E(T), \  b(0)=a_0,\  \tau(0)=\widetilde \theta_0\ \hbox{and}\ \|(b,\tau)\|_{E(T)}\leq c\}$$
 for some suitable small positive constants $c$, which will be determined shortly.
 Next, given $(b,\tau)\in K(T)$, we define the mapping
 $$\Phi(b,\tau):=(a,\widetilde\theta),$$
 where $(a,\widetilde\theta)$ is defined as the unique solution of the corresponding linearized problem of \eqref{5.2}
 \begin{equation}\label{5.2a}
	\left \{
	\begin{aligned}
		&\partial_ta-\kappa_1\Delta a+\kappa_1\Delta \widetilde \theta=F(b,\tau),\\
		&\kappa_2\partial_t\widetilde\theta-(\kappa_1^2+\bar\kappa_3)\Delta\widetilde\theta+\kappa_1^2\Delta a=G(b,\tau),\\
		&(a,\widetilde \theta)|_{t=0}=(a_0,\widetilde\theta_0),
	\end{aligned} \right.
\end{equation}
where
\begin{align}
	\begin{split}\label{5.3a}
		&F(b,\tau)=-2\kappa_1\frac{\tau+1}{1+b}|\nabla b |^2 +2\kappa_1 \nabla b\cdot \nabla \tau-\kappa_1b\Delta\tau+\kappa_1 \tau \Delta b,\\
		&G(b,\tau)=2 \kappa_1^2 \frac{(\tau+1)^2}{(1+b)^2}|\nabla b|^2 - (3\kappa_1^2 + \kappa_1 \kappa_2) \frac{(\tau+1)}{1+b} \nabla b\cdot \nabla  \tau		\\
		&\quad -\kappa_1^2 \frac{\tau^2 + 2\tau}{1+b}\Delta b + \kappa_1^2 \frac{b}{1+b} \Delta b
	+\bar\kappa_3 b \Delta\tau+\kappa_1^2\tau \Delta \tau\\
		&\quad +(1+b)\nabla\cdot(\widetilde\kappa_3(\tau)\nabla\tau) + \kappa_1(\kappa_1+\kappa_2)|\nabla\tau|^2.
	\end{split}
\end{align}

Following Propositions \ref{p5.2}, we easily obtain that
\begin{align}\label{phi}
	\begin{split}
	\|\Phi(b,\tau)\|_{E(T)}	\leq C	\Big( \|(a_0,\widetilde\theta_0)\|_{\dot B_{2,1}^{3/2}}+\|F(b,\tau)\|_{L^{1}_T(\dot B_{2,1}^{3/2})}+\|G(b,\tau)\|_{L^{1}_T(\dot B^{3/2}_{2,1})}\Big).
\end{split}
\end{align}
In order to prove that $\Phi(K(T))\subset K(T)$ under the smallness condition on $a_0$ and $\widetilde\theta_0$, one needs to bound the right side of \eqref{phi}. We ignore $\kappa_1$, $\kappa_2$, and $\bar\kappa_3$, as they are fixed constants.

For the first term in \eqref{5.3a},  we rewrite it as
$$ \frac{\tau + 1}{1+b}|\nabla b|^2 = m_1(b)|\nabla b|^2 ( \tau + 1) + |\nabla b|^2 ( \tau+1), $$
where $m_1(b):=\frac{1}{1+b}-1$ satisfying $m_1(0)=0.$
 By Lemma 1.6 in \cite{Da6}, the continuity of the product in Besov spaces (Chapter 2 in
 \cite{BCD}),  we get
\begin{align*}
	\begin{split}
		\left\|\frac{\tau+1}{1+b}|\nabla b|^2 \right\|_{L^{1}_T(\dot B_{2,1}^{3/2})}
			&\leq C \left( 1 + \|b\|_{L^{\infty}_T(\dot B_{2,1}^{3/2})}\right) \|\nabla b\|^2_{L^{2}_T(\dot B_{2,1}^{3/2})}
			\left( \|\tau\|_{L^{\infty}_T(\dot B_{2,1}^{3/2})} + 1 \right)\\
		&\leq  C\left(1+c\right)^2 c^2.
 \end{split}
	\end{align*}
	Similarly,  we have
\begin{align*}
	\begin{split}
		&\|\nabla b\cdot\nabla \tau\|_{L^{1}_T(\dot B_{2,1}^{3/2})}
		\leq  Cc^2,\\
		&\|b\cdot\Delta \tau\|_{L^{1}_T(\dot B_{2,1}^{3/2})}+\|\tau\Delta b\|_{L^{1}_T(\dot B_{2,1}^{3/2})}
		\leq  Cc^2.
 \end{split}
	\end{align*}
	Combining the above estimates, we find that
	$$\|F(b,\tau)\|_{L^{1}_T(\dot B_{2,1}^{3/2})}\leq C\left(1+c\right)^2 c^2.$$
	The  terms in  $G$ can be bounded in essentially the same way. For the first term in $G$, we rewrite it as
$$ \frac{(\tau+1)^2}{(1+b)^2} |\nabla b|^2 = m_2(b)|\nabla b|^2(\tau+1)^2+|\nabla b|^2 (\tau+1)^2, $$
where $m_2(b):=\frac{1}{(1+b)^2}-1$ satisfying $m_2(0)=0.$ We then infer that
\begin{align*}
	\begin{split}
		\left\|\frac{(\tau+1)^2}{(1+b)^2}|\nabla b |^2 \right\|_{L^{1}_T(\dot B_{2,1}^{3/2})}
		&\leq C \left( 1 + \|b\|_{L^{\infty}_T(\dot B_{2,1}^{3/2})}\right) \|\nabla b\|^2_{L^{2}_T(\dot B_{2,1}^{3/2})} \left( \|\tau\|_{L^{\infty}_T(\dot B_{2,1}^{3/2})}+1 \right)^2\\
		&\leq  C\left(1+c\right)^3c^2.
 \end{split}
	\end{align*}
The term $\frac{\tau+1}{1+b} \nabla b \cdot \nabla \tau$ is handled the same as $\frac{\tau+1}{1+b} |\nabla b|^2$. The third term in $G$ can be rewritten as
$$ \frac{\tau^2+2\tau}{1+b} \Delta b = m_1(b) \tau \Delta b (\tau+2) + \tau \Delta b (\tau + 2), $$
which is estimated in the same way. The fourth term of $G$ is in fact $-m_1(b) \Delta b$, so that
$$ \left\| \frac{b}{1+b} \Delta b \right\|_{L^1_T(\dot B_{2,1}^{3/2})} \leq C \| b \|_{L^\infty_T(\dot B_{2,1}^{3/2})}
\| \Delta b \|_{L^1_T(\dot B_{2,1}^{3/2})} \leq C c^2. $$
Likewise
$$ \| b \Delta \tau \|_{L^1_T(\dot B_{2,1}^{3/2})} + \| \tau \Delta \tau \|_{L^1_T(\dot B_{2,1}^{3/2})} \leq C c^2. $$
The seventh term of $G$ becomes
$ (1+b) \widetilde\kappa_3'(\tau) |\nabla \tau|^2 + (1+b) \widetilde\kappa_3(\tau) \Delta \tau$, where $\widetilde\kappa_3(0)=0$ and $\widetilde\kappa_3'$
is bounded by assumption. The $L^1_T(\dot B_{2,1}^{3/2})$-norm for this term is controlled by $C(1+c) c^2$. The last term in $G$ is similarly bounded by $C c^2$. Thus, we have
	$$\|G(b,\tau)\|_{L^{1}_T(\dot B_{2,1}^{3/2})}\leq C\left(1+c\right)^3c^2.$$
Finally, combining the above  estimate and the one for $F$ with \eqref{phi}, we obtain that
\begin{align}\label{phi2}
	\begin{split}
	\|\Phi(b,\tau)\|_{E(T)}	\leq C	 \|(a_0,\widetilde\theta_0)\|_{\dot B_{2,1}^{3/2}}+C(1+c)^3c^2.
\end{split}
\end{align}
This implies $\Phi(K(T))\subset K(T)$ provided that
\begin{align}\label{small1}
c\leq \min \{1, \frac{1}{16C}\}\quad \hbox {and}\quad \|(a_0,\widetilde\theta_0)\|_{\dot B_{2,1}^{3/2}}\leq \frac{1}{2C}c.
\end{align}
Next, we will prove that for any $T>0$, the map $\Phi (b,\tau)$ is contractive on $K(T)$.
Indeed, for $(v_i,\tau_i)\in K(T)$, let $(a_i,\widetilde\theta_i)=\Phi(v_i,\tau_i)$	 for $i=1,2.$ Moreover, we set $\bar a=a_1-a_2$ and $\bar \theta=\widetilde \theta_1-\widetilde \theta_2.$ Then $(\bar a, \bar\theta)$ satisfies the equation
\begin{equation}\label{4.16}
	\left \{
	\begin{aligned}
		&\partial_t\bar a-\kappa_1\Delta \bar a+\kappa_1\Delta \bar \theta=\delta F,\\
		&\kappa_2\partial_t\bar \theta-(\kappa_1^2+\bar\kappa_3)\Delta\bar \theta+\kappa_1^2\Delta \bar a=\delta G,\\
		&(\bar  a,\bar \theta)|_{t=0}=(0,0),
	\end{aligned} \right.
\end{equation}
where $\delta F = F(b_1, \tau_1) - F(b_2, \tau_2)$ and $\delta G = G(b_1, \tau_1) - F(b_2, \tau_2)$. Applying Proposition \ref{p5.2} yields:
\begin{align}\label{delta}
	\begin{split}
	\|(\bar a,\bar\theta)\|_{L^\infty_T(\dot B_{2,1}^{3/2})}+\|(\Delta\bar a, \Delta\bar\theta)\|_{L^1_T(\dot B_{2,1}^{3/2})}
	\leq C	\left( \|\delta F\|_{L^{1}_T(\dot B_{2,1}^{3/2})}+\|\delta G\|_{L^{1}_T(B^{3/2}_{2,1})}\right).
\end{split}
\end{align}
Now, $\delta F$ can be rewritten as follows
\begin{align}
	\begin{split}\label{4.18}
		\delta F&=-2\kappa_1\left(\frac{1}{1+b_1}-\frac{1}{1+b_2}\right)|\nabla b_2|^2 (\tau_2+1)-2\kappa_1\frac{1}{1+b_1}\nabla \delta b\cdot \nabla b_2(\tau_2+1)\\
		&\quad-2\kappa_1\frac{1}{1+b_1}\nabla b_1\cdot \nabla \delta b\ (\tau_2+1)-2\kappa_1\frac{1}{1+b_1}|\nabla b_1|^2 \delta \tau+2\kappa_1\nabla \delta b\cdot\nabla\tau_2\\
		&\quad+2\kappa_1\nabla b_1 \cdot\nabla\delta\tau-\kappa_1\delta b\Delta \tau_2-\kappa_1 b_1\Delta \delta\tau+\kappa_1\delta \tau\Delta b_2+\kappa_1\tau_1\Delta \delta b,
		 	\end{split}
\end{align}
where $\delta b=b_1- b_2, \delta \tau=\tau_1-\tau_2$. Moreover, for the first term we also have
\begin{equation}\label{4.20}
\frac{1}{1+b_2}-\frac{1}{1+b_1} =\frac{1}{(1+b_2)(1+b_1)}\delta b =(m_1(b_1)+1)(m_1(b_2)+1)\delta b.
\end{equation}
Then we can estimate the terms in $\delta F$ analogously to the terms in \eqref{5.3a} to obtain
\begin{align*}
\|\delta F\|_{L^{1}_T(\dot B_{2,1}^{3/2})}&\leq C (\|b_1\|_{L^{\infty}_T(\dot B_{2,1}^{3/2})}+1)(\|b_2\|_{L^{\infty}_T(\dot B_{2,1}^{3/2})}+1)\|\delta b\|_{L^{\infty}_T(\dot B_{2,1}^{3/2})}\|\nabla b_2\|^2_{L^{2}_T(\dot B_{2,1}^{3/2})}\\
&\quad\times(\|\tau_2\|_{L^{\infty}_T(\dot B_{2,1}^{3/2})}+1)\\
&\quad+C (\|b_1\|_{L^{\infty}_T(\dot B_{2,1}^{3/2})}+1)\|\nabla\delta b\|_{L^{2}_T(\dot B_{2,1}^{3/2})}\|\nabla b_2\|_{L^{2}_T(\dot B_{2,1}^{3/2})}(\|\tau_2\|_{L^{\infty}_T(\dot B_{2,1}^{3/2})}+1)\\
&\quad+C (\|b_1\|_{L^{\infty}_T(\dot B_{2,1}^{3/2})}+1)\|\nabla\delta b\|_{L^{2}_T(\dot B_{2,1}^{3/2})}\|\nabla b_1\|_{L^{2}_T(\dot B_{2,1}^{3/2})}(\|\tau_2\|_{L^{\infty}_T(\dot B_{2,1}^{3/2})}+1)\\
&\quad+C (\|b_1\|_{L^{\infty}_T(\dot B_{2,1}^{3/2})}+1)\|\nabla b_1\|^2_{L^{2}_T(\dot B_{2,1}^{3/2})}\|\delta\tau\|_{L^{\infty}_T(\dot B_{2,1}^{3/2})}\\
&\leq C(1+c)^3c\big(\|\delta b\|_{E(T)}+\|\delta\tau\|_{E(T)}\big).
\end{align*}
A similar methodology is applied for the terms in $\delta G$. We write
\begin{align*}
\delta G &= \kappa_1^2 J_1 -(3\kappa_1^2+\kappa_1 \kappa_2) J_2 -\kappa_1^2 J_3 + \kappa_1^2 J_4 \\
&\quad + \bar\kappa_3 J_5 + \kappa_1^2 J_6+ J_7 + \kappa_1 (\kappa_1 + \kappa_2) J_8.
\end{align*}
Each of the $J_i$ terms correspond to the difference operator $\delta$ applied to each respective term in the expression for $G$ in \eqref{5.3a}. Specifically, we have
\begin{align*}
J_1 &= \left( \frac{1}{1+b_1} - \frac{1}{1+b_2} \right) \left( \frac{(\tau_2 + 1)^2}{1+b_2} + \frac{(\tau_2+1)^2}{1+b_1} \right) |\nabla b_2|^2 \\
&\quad + \frac{\delta \tau (\tau_1 + \tau_2 + 2)}{(1+b_1)^2} |\nabla b_2|^2
+ \frac{(\tau_1+1)^2}{(1+b_1)^2} \nabla \delta b \cdot (\nabla b_1 + \nabla b_2),
\end{align*}
\begin{align*}
J_2 &= \left( \frac{1}{1+b_1} - \frac{1}{1+b_2} \right)(\tau_2 + 1) \nabla b_2 \cdot \nabla \tau_2
+ \frac{\delta \tau}{1+b_1} \nabla b_2 \cdot \nabla \tau_2 \\
&\quad + \frac{\tau_1 + 1}{1+b_1} \nabla \delta b \cdot \nabla \tau_2
+ \frac{\tau_1 + 1}{1+b_1} \nabla b_1 \cdot \nabla \delta \tau_2,
\end{align*}
$$
J_3 = \left( \frac{1}{1+b_1} - \frac{1}{1+b_2} \right) (\tau_2^2+2\tau_2) \Delta b_2
+ \frac{\delta \tau (\tau_1+\tau_2+2)}{1+b_1} \Delta b_2
+ \frac{\tau_1^2 + 2 \tau_1}{1+b_1} \Delta \delta b,
$$
$$
J_4 =\left( \frac{b_1}{1+b_1} - \frac{b_2}{1+b_2} \right) \Delta b_2 + \frac{b_1}{1+b_1} \Delta \delta b,
$$
$$
J_5 = \delta b \Delta \tau_2 + b_1 \Delta \delta \tau,
$$
$$
J_6 = \delta \tau \Delta \tau_2 + \tau_1 \Delta \delta \tau,
$$
\begin{align*}
J_7 &= \delta b (\widetilde\kappa_3(\tau_2) \Delta \tau_2 + \widetilde\kappa_3'(\tau)|\nabla \tau_2|^2) \\
&\quad + (1+b_1) (\widetilde\kappa_3(\tau_1)-\widetilde\kappa_3(\tau_2)) \Delta \tau_2
+ (1+b_1)\widetilde\kappa_3(\tau_1) \Delta \delta \tau \\
&\quad + (1+b_1)(\widetilde\kappa_3'(\tau_1) - \widetilde\kappa_3'(\tau_2)) |\nabla \tau_2|^2
+(1+b_1)\widetilde\kappa_3'(\tau_1) \nabla \delta \tau \cdot (\nabla \tau_1 + \nabla \tau_2),
\end{align*}
$$
J_8 = \nabla \delta \tau \cdot (\nabla \tau_1 + \nabla \tau_2).
$$
The estimate for each $J_i$ is similar to those already established, keeping in mind the continuity of products in Besov spaces, Lemma 1.6 in \cite{Da6},
and the identity \eqref{4.20} to control several terms in $J_1$, $J_2$, $J_3$, and $J_4$. The only subtle point is that, to estimate $J_7$, we need to invoke the mean-value theorem to write
$$ |\widetilde\kappa_3(\tau_1)-\widetilde\kappa_3(\tau_2)| \leq C |\delta \tau| \quad\quad\quad
|\widetilde\kappa_3'(\tau_1)-\widetilde\kappa_3'(\tau_2)| \leq C |\delta \tau|, $$
where $C$ above depends on the upper bound for $|\widetilde\kappa_3'|$ and $|\widetilde\kappa_3''|$. Since this upper bound exists by assumption, we can infer that
$$
\|\delta G\|_{L^{1}_T(\dot B_{2,1}^{3/2})}\leq C(1+c)^5c\left(\|\delta b\|_{E(T)}+\|\delta\tau\|_{E(T)}\right).
$$
Therefore,
$$
	\|(\bar a,\bar\theta)\|_{L^\infty_T(\dot B_{2,1}^{3/2})}+\|(\Delta\bar a, \Delta\bar\theta)\|_{L^1_T(\dot B_{2,1}^{3/2})}
	\leq C(1+c)^5c\left(\|\delta b\|_{E(T)}+\|\delta\tau\|_{E(T)}\right).
$$
If we additionally assume that $c \leq \frac{1}{64C}$, we then have that for all $T>0$
$$
	\|(\bar a,\bar\theta)\|_{L^\infty_T(\dot B_{2,1}^{3/2})}+\|(\Delta\bar a, \Delta\bar\theta)\|_{L^1_T(\dot B_{2,1}^{3/2})}
	\leq \frac{1}{2}\left(\|\delta b\|_{E(T)}+\|\delta\tau\|_{E(T)}\right).
$$
Thus, $\Phi$ is contractive as a mapping from $E(T)$ to $E(T)$. The proof of Theorem \ref{t1.3} follows from Lemma \ref{KP}.

\section*{Acknowledgement}
The first author was partially supported by Zhejiang Province Science fund (LY21A010009). The second author was partially supported by NSFC (12171097).
The third author was partially supported by
NSF DMS-2012333 and DMS-2108209.

\end{document}